\documentclass[a4paper]{article}
\usepackage[margin=1in]{geometry}

\usepackage{amsmath,amsthm,amssymb,bbm,hyperref,enumitem,graphicx,caption,subcaption}
\usepackage{natbib}
\usepackage{xcolor}
\usepackage{float}

\hypersetup{colorlinks,citecolor=purple,urlcolor=black}%

\newcommand{\N}{\mathbb{N}}

\newcommand{\R}{\mathbb{R}}

\newcommand{\E}{\mathbb{E}}

\newcommand{\cX}{\mathcal{X}}

\newcommand{\cH}{\mathcal{H}}
\newcommand{\cJ}{\mathcal{J}}
\newcommand{\cA}{\mathcal{A}}

\newcommand{\cI}{\mathcal{I}}

\newcommand{\ind}{\mathbbm{1}}
\newcommand{\RN}[1]{\textup{\uppercase\expandafter{\romannumeral#1}}}

\newcommand{\iid}{\overset{\mathrm{iid}}{\sim}}

\newcommand{\given}{\,|\,}

\makeatletter
\newcommand*\bigcdot{\mathpalette\bigcdot@{.5}}
\newcommand*\bigcdot@[2]{\mathbin{\vcenter{\hbox{\scalebox{#2}{$\m@th#1\bullet$}}}}}
\makeatother

\DeclareMathOperator\Var{Var}

\newtheorem{theorem}{Theorem}
\newtheorem{lemma}[theorem]{Lemma}

\newtheorem{example}[theorem]{Example}

\makeatletter
\def\underbar{\underline}
\makeatother

\title{The augmented van Trees inequality}
\author{Elliot Young}
\usepackage{authblk}
\author[1]{{Elliot H.\ Young}
\thanks{ey244@cam.ac.uk}}
\affil{Statistical Laboratory, University of Cambridge, UK}
\date{March 2026}

\begin{document}

\maketitle

\begin{abstract} 
We introduce an augmented form of the van Trees inequality, that yields uniformly tighter lower bounds on the minimax squared Bayes risk of estimators compared with the classical van Trees inequality. Our augmented inequality also accommodates prior distributions whose densities need not vanish at the boundaries of their supports. We demonstrate how this refinement can be utilized for elementary proofs of a number of minimax lower bounds for nonparametric estimands, that also often attain sharper constants than those obtained by the alternative Le Cam convergence of experiments theory and the classical van Trees inequality, and in some cases obtain exact constants. As an example, our augmented van Trees inequality can be used to obtain the asymptotic minimax pointwise mean squared error when estimating the regression function in the model with normal errors: when the regression function is univariate and differentiable with Lipschitz derivative we obtain this quantity up to a constant factor of $1.37$; and in the high dimensional regime with a H\"older smooth regression function of smoothness $\beta\in(0,2]$ we obtain exact constants. Both these results do not follow from an application of the classical van Trees inequality. The flexibility of our augmented van Trees inequality accommodates lower bounds for models beyond Gaussianity, loss functions beyond the squared error loss, and we are also able to incorporate this augmentation into generalized versions of the van Trees inequality for irregular models.
\end{abstract}

\section{Introduction}\label{sec:intro}

The van Trees inequality~\citep{vantrees} is a well-celebrated result that offers a lower bound on the Bayesian risk of an estimator under a given prior, and is commonly viewed as a Bayesian Cram\'{e}r--Rao bound.
Its relevance to statistical theory was popularized by \citet{gill}, who demonstrated how the inequality can be leveraged to obtain minimax lower bounds. In particular, for $\sqrt{n}$-consistent parameter estimation problems, the van Trees inequality yields asymptotically sharp risk lower bounds, recovering classical results such as the Cram\'{e}r--Rao lower bound in parametric models as well as local asymptotic minimaxity results in $\sqrt{n}$-consistent semiparametric settings~\citep{vandervaart, gassiat}.

In nonparametric estimation, the van Trees inequality can provide a remarkably simple strategy to derive minimax lower bounds that circumvents the sophisticated convergence of experiments theory of~\citet{hajek1, lecam, assouad}. 
In a variety of nonparametric problems however these more sophisticated constructions can be utilized to provide tighter minimax lower bounds in terms of constants than those obtained by the van Trees inequality.

In this paper we develop an \emph{augmented van Trees inequality} that both expands the class of admissible priors and consequently yields strict improvements on the lower bound for the worst-case risk. 
This extension when applied to minimax lower bound constructions attempts to obtain the best of both worlds; providing a remarkably simple and off-the-shelf methodology as one obtains with the van Trees inequality, whilst also attaining tighter constants, which are often sharper than those obtained by the more involved convergence of experiments theory described above. We study as an example that of pointwise H\"older function estimation (see Section~\ref{sec:minimax} to follow), where we are in particular interested in attaining (near-)sharp constants in our minimax lower bounds. 

\subsection{The van Trees inequality, and outline of our contributions}

To fix notation, consider a parametric model $(P_t)_{t\in T}$ indexed by a compact interval $T=[t_1,t_2]$, with $-\infty < t_1 < t_2 < \infty$. Assume that $P_t \ll \nu$ for all $t \in T$ for some $\sigma$-finite measure $\nu$ on $(\mathcal{X}, \mathcal{A})$, and let $p(\cdot,t)$ denote the density of $P_t$ with respect to $\nu$, which we assume is absolutely continuous. 
The Fisher information we define as
\begin{equation*}
    \cI(t):=\int_\cX\biggl(\frac{\partial_t p(x,t)}{p(x,t)}\biggr)^2p(x,t)\,d\nu(x),
\end{equation*}
which we assume to be finite for all $t \in T$. 
Given an arbitrary estimator $\hat{t}(\boldsymbol{X})$ depending on data $\boldsymbol{X}\sim P_t$ we are primarily interested in obtaining a lower bound on the worst case risk
\begin{equation}\label{eq:Bayes-risk}
    \sup_{t\in T}\E_{P_t}\big[(\hat{t}(\boldsymbol{X})-t)^2\big] = \sup_{t\in T}\int_\cX (\hat{t}(x)-t)^2p(x,t)\,d\nu(x).
\end{equation}
The classical van Trees approach for minimax lower bounds~\citep[e.g.][]{tsybakov} bounds this quantity from below by the Bayes risk \begin{equation}\label{eq:sup-leq-prior}
    \sup_{t\in T}\E_{P_t}\big[(\hat{t}(\boldsymbol{X})-t)^2\big]
    \geq
    \int_T\E_{P_t}\big[(\hat{t}(\boldsymbol{X})-t)^2\big]\mu(t)\,dt
    =
    \int_T\int_\cX (\hat{t}(x)-t)^2p(x,t)\,d\nu(x)\,\mu(t)\,dt,
\end{equation}
for a suitable prior distribution on $T$ with absolutely continuous density $\mu$ with respect to Lebesgue measure, and finite prior information $\cJ(\mu):=\int_T\frac{(\mu')^2}{\mu}$. Note that here, and throughout the paper, we will take by default $1/0=0$, and so can write e.g.~$\int_T\frac{(\mu')^2}{\mu}\ind_{\{\mu(\cdot)>0\}}=\int_T\frac{(\mu')^2}{\mu}$.

The classical van Trees inequality~\citep{vantrees} offers a lower bound for~\eqref{eq:sup-leq-prior} provided the prior density is absolutely continuous and vanishes at the boundary of the support of $T$.  
For any statistic $\hat{t}(\boldsymbol{X})$ and any prior $\mu$ with $\mu(t_1)=\mu(t_2)=0$, 
\begin{equation}\label{eq:vT}
    \int_T\E_{P_t}\big[(\hat{t}(\boldsymbol{X})-t)^2\big]\mu(t)\,dt
    \geq
    \frac{1}{\int_T\cI(t)\mu(t)\,dt+\cJ(\mu)}.
\end{equation}
For $T=[-1,1]$ the infimum of $\cJ(\mu)$ over such priors is $\pi^2$, achieved by $\mu(t)=\cos^2({\pi t}/{2})$. Further, if $\cI(\cdot)=\cI$ constant (e.g.~as in the settings of Section~\ref{sec:minimax}) then the van Trees inequality simplifies to
\begin{equation}\label{eq:vTcurve}
    \sup_{\mu}\int_{-1}^1\E_{P_t}\big[(\hat{t}(\boldsymbol{X})-t)^2\big]\mu(t)\,dt
    \geq
    \frac{1}{\cI+\pi^2},
\end{equation}
with supremum taken over all valid priors as described above.

Our \emph{augmented van Trees} inequality extends the van Trees inequality in two key respects:
\begin{enumerate}[label=(\roman*)]
\item We allow for priors that need not vanish at the boundary of $T$; and
\item Our lower bound involves an 
\emph{augmentation function}, that when carefully chosen guarantees uniformly sharper lower bounds on the worst-case risk~\eqref{eq:Bayes-risk}. 
\label{item:beter-bounds}
\end{enumerate}
Regarding~\ref{item:beter-bounds}, we introduce two helpful representative bounds on the worst case risk that can be derived from our augmented van Trees inequality:
\begin{enumerate}
    \item\label{item:AVT1}
    Augmented van Trees 1 (AVT1): A simple to use, albeit non-optimal augmentation is 
    \begin{equation}\label{eq:AVT-simple}
    \sup_{t\in[-1,1]}\E_{P_t}\big[(\hat{t}(\boldsymbol{X})-t)^2\big]\geq
    \max\Biggl(
    \frac{1}{(\sqrt{\cI}+1)^2}
    ,\,
    \frac{1}{\cI+\pi^2}\Biggr).
    \end{equation} 
    \item\label{item:AVT2}
    Augmented van Trees 2 (AVT2): A sharper but more involved lower bound is
    \begin{equation}\label{eq:AVT}
    \sup_{t\in[-1,1]}\E_{P_t}\big[(\hat{t}(\boldsymbol{X})-t)^2\big]\geq\frac{1}{\inf_{m>0}(m+1)^2\, \bigl\{\hspace{0cm}_2F_1\bigl(-\tfrac{1}{2},\tfrac{m}{2},\tfrac{m}{2}+1;-\tfrac{\cI}{m^2}\bigr)\bigr\}^2},
\end{equation}
where $_2F_1$ is the hypergeometric function. 
\end{enumerate}
See e.g.~\citet{hypergeometric} for more details on the hypergeometric function. 
Both of these lower bounds, alongside the lower bound obtained via the classical van Trees inequality, are given in Figure~\ref{fig:1/(1+x)}. 

\begin{figure}
    \centering
    \includegraphics[width=0.75\linewidth]{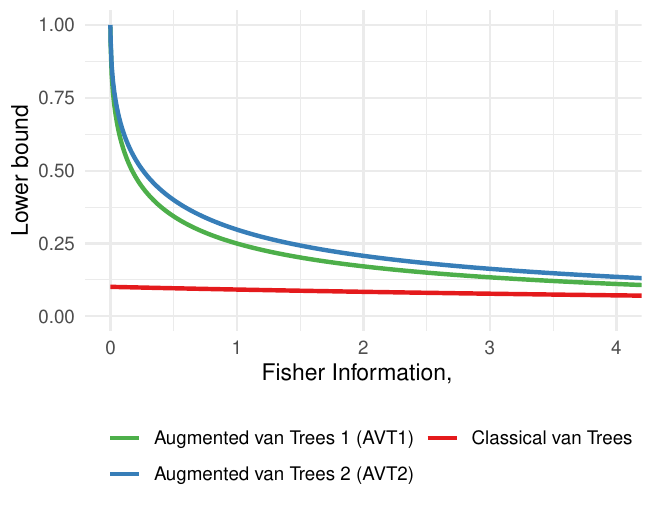}
    \caption{The lower bound for $\sup_{t\in T}\E_{P_t}\big[(\hat{t}(\boldsymbol{X})-t)^2\big]$ as given by the van Trees inequality (with optimal prior)~\eqref{eq:vTcurve} and the augmented van Trees inequality (with approximately optimal prior)~\eqref{eq:AVT}.}
    \label{fig:1/(1+x)}
\end{figure}

The remainder of the paper is organized as follows. In Section~\ref{sec:gvt} we state and prove the augmented van Trees inequality (Theorem~\ref{thm:gvt}). Section~\ref{sec:minimax} applies this result to H\"older function estimation, yielding sharper minimax lower bounds, including in some settings exact constants. In Section~\ref{sec:takatsu} we show how our augmentation mechanism may also be applied to the recently introduced generalized van Trees inequality~\citep{gvt}, outlining the flexibility and applicability of our augmentation scheme. 

\section{The Augmented van Trees inequality}\label{sec:gvt}

Consider the parametric model $(P_t)_{t\in T}$ in Section~\ref{sec:intro}. The augmented van Trees inequality below bounds the squared Bayes risk for any (Lebesgue) prior, and involves an additional auxiliary~\emph{augmentation function} $\alpha:T\to\R$.

\begin{theorem}[Augmented van Trees inequality]\label{thm:gvt}
    Assume that:
    \begin{enumerate}[label=(\roman*)]
        \item $p:\cX\times T\to [0,\infty)$ is a measurable function such that $p(\cdot,t)$ is a Lebesgue density for each $t\in T$, the function $t\mapsto p(x,t)$ is absolutely continuous for $\nu$-almost all $x$, and the Fisher information
        $$\cI(t):=\int_\cX\bigg(\frac{\partial_t p(x,t)}{p(x,t)}\bigg)^2p(x,t)\,d\nu(x),$$
        is finite and satisfies $\int_T\cI(t)\,dt<\infty$.
        \label{ass:DGM}
        \item The prior $\mu$ is a Lebesgue density on $T$ and is an absolutely continuous function.
        \item The augmentation function $\alpha:T\to\R$ is absolutely continuous with $\alpha(t_1)=\alpha(t_2)=0$, and
        \begin{gather*}
            \max\biggl\{\int_T \alpha,
            \;
            \int_T \frac{\alpha^2}{\mu},
            \;
            \int_T \frac{(\alpha')^2}{\mu}
            \biggr\}
            <\infty.
        \end{gather*}
    \end{enumerate}
    Then for any measurable function $\hat{t}:\cX\to\R$, the Bayes risk satisfies
    \begin{equation}\label{eq:gvt-thm}
        \int_T\E_{P_t}\big[(\hat{t}(\boldsymbol{X})-t)^2\big]
        \mu(t)\,dt
        \geq
        \frac{\big(\int_T\alpha\big)^2}{\int_T\frac{\cI \alpha^2+(\alpha')^2}{\mu}
        }.
    \end{equation}
\end{theorem}
Taking $\alpha=\mu$ recovers the classical van Trees inequality~\eqref{eq:vT}. The incorporation of the auxiliary augmentation function $\alpha$ allows for tighter lower bounds, obtained for example by optimizing~\eqref{eq:gvt-thm} over a suitable class of augmentation functions $\alpha$. 
In contrast to the classical van Trees inequality and its generalizations~\citep{vantrees, gassiat, gvt} we do not require the prior density $\mu$ to vanish at the boundary of $T$, with that slack taken up instead by the augmentation function $\alpha$. 
In fact, this allows us to obtain a lower bound on the squared Bayes risk for any absolutely continuous prior density. 
This additional flexibility allows us to concentrate a more substantial proportion of the prior's mass near the boundary of 
$T$, which typically corresponds to the parameter values that are most difficult to distinguish between for a given statistic, improving the tightness of the bound \eqref{eq:sup-leq-prior} (see Figure~\ref{fig:batman} later for an example). Consequently, applying the augmented van Trees inequality of Theorem~\ref{thm:gvt} together with \eqref{eq:sup-leq-prior}, and optimizing over valid priors, yields the following lower bound on the worst-case error.
\begin{theorem}\label{prop:gvt-simple}
    Consider the class of distributions $(P_t)_{t\in T}$ satisfying Assumption~\ref{ass:DGM} of Theorem~\ref{thm:gvt}, and suppose $\cI$ is absolutely continuous on $T$. 
    Let $\mathcal{A}$ be any class of augmentation functions $\alpha:T\to\R$ that are:~(i) absolutely continuous on $T$ with weak derivative $\alpha'$ such that $(\alpha')^2$ is absolutely continuous on $T$; and~(ii) satisfy $\alpha(t_1)=\alpha(t_2)=0$. 
    Then 
    \begin{equation}\label{eq:stronger-bound}
        \sup_{t\in T}\E_{P_t}\big[(\hat{t}(\boldsymbol{X})-t)^2\big]
        \geq 
        \sup_{\alpha\in\mathcal{A}}\Biggl(\frac{\int_T\alpha }{\int_T\sqrt{\cI \alpha^2+(\alpha')^2}}\Biggr)^2.
    \end{equation}
\end{theorem}
Theorem~\ref{prop:gvt-simple} is a simple consequence of the Cauchy--Schwarz inequality; 
for a given $\alpha\in\mathcal{A}$ the prior that maximizes the lower bound~\eqref{eq:gvt-thm} is
\begin{equation}\label{eq:mu-opt-g}
    \mu(t)\propto \sqrt{\cI(t)\alpha^2(t)+\alpha'(t)^2}.
\end{equation}
Note therefore taking any class of augmentation functions $\cA$ that contains the function $\alpha(t)=\cos^2(\pi\hspace{0.04em}t/2)$ guarantees that the corresponding lower bound~\eqref{eq:stronger-bound} dominates the classical van Trees inequality. 
Whilst the variational problem corresponding to the supremum in \eqref{eq:stronger-bound} over all suitable augmentation functions yields a rather involved Euler--Lagrange equation, we may still obtain reasonable bounds by selecting a sensible ansatz class of augmentation functions, such as the two cases below.
\begin{example}[Augmented van Trees 1]\label{ex:aVT1}
    Suppose $\cI(\cdot)=\cI$ is constant and $T=[-1,1]$. Taking the class of functions $\cA=\bigl\{\alpha(t)=\ind_{\{|t|\leq 1-\delta\}}+\bigl(\frac{1-t}{\delta}\bigr)\ind_{\{1-\delta<|t|\leq 1\}}\,:\,\delta\in(0,1)\bigr\}$. 
    Then $\int_T\alpha\geq (1-\delta)$ and $\int_T\sqrt{\cI\alpha^2+(\alpha')^2}\leq(1-\delta)\sqrt{\cI}+\sqrt{\delta^2\cI+1}\leq \sqrt{\cI}+1$. Taking $\delta\searrow0$ we obtain
    \begin{equation*}
        \sup_{t\in [-1,1]}\E_{P_t}\big[(\hat{t}(\boldsymbol{X})-t)^2\big]
        \geq
        \frac{1}{(\sqrt{\cI}+1)^2}.
    \end{equation*}
\end{example}

\begin{example}[Augmented van Trees 2]\label{ex:hypergeo}
    Suppose again $\cI(\cdot)=\cI$ is constant and $T=[-1,1]$. Consider the class of augmentation functions $\mathcal{A}=\big\{\alpha(t)=(1-|t|)^m:m>0\big\}$. Then (see Appendix~\ref{appsec:hypergeo}) the inequality~\eqref{eq:stronger-bound} becomes
    \begin{equation*}
        \sup_{t\in [-1,1]}\E_{P_t}\big[(\hat{t}(\boldsymbol{X})-t)^2\big]
        \geq 
        \frac{1}{\inf_{m>0}(m+1)^2\bigl\{\hspace{0cm}_2F_1\bigl(-\tfrac{1}{2},\tfrac{m}{2},\tfrac{m}{2}+1,-\tfrac{\cI}{m^2}\bigr)\bigr\}^2}.
    \end{equation*}
\end{example}
The two lower bounds obtained above are plotted in Figure~\ref{fig:1/(1+x)}.

\subsection{Extensions to other loss functions} 
The van Trees inequality is typically restricted to the squared error risk. We extend the (augmented) van Trees inequality to general $L_p$ loss functions below. 
\begin{theorem}\label{thm:genrealized-loss}
Adopt the setup of Theorems~\ref{thm:gvt}~and~\ref{prop:gvt-simple}. Suppose $p,q>1$ satisfies $1/p+1/q=1$. Also suppose the conditional score function $\rho_t(x) := \frac{\partial_tp(x,t)}{p(x,t)}$ is such that $\E_{P_t}[|\rho_t({\boldsymbol{X}})|^q]$ is finite and integrable. Then
    \begin{equation*}
        \sup_{t\in T}\E_{P_t}\bigl[|\hat{t}(\boldsymbol{X})-t|^p\bigr]
        \geq
        \sup_{\alpha\in\cA}\Biggl(\frac{\bigl|\int_T\alpha(t)\,dt\bigr|}{\int_T
        \bigl\{\E_{P_t}\bigl[|\alpha'(t)+\alpha(t)\rho_t({\boldsymbol{X}})|^q\bigr]\bigr\}^{1/q}
        \,dt}\Biggr)^p,
    \end{equation*}
    where the supremum is take over all augmentation functions $\cA$ as in Theorem~\ref{prop:gvt-simple}.
\end{theorem}
Theorem~\ref{prop:gvt-simple} is recovered by taking $p=q=2$, where
$$\E_{P_t}\bigl[\{\alpha'(t)+\alpha(t)\rho_t(\boldsymbol{X})\}^2\bigr]=\alpha'(t)^2+\alpha^2(t)\E_{P_t}[\rho_t^2(\boldsymbol{X})]=\alpha'(t)^2+\alpha^2(t)\cI(t).$$

\section{Applications to minimax function estimation}\label{sec:minimax}

We apply the augmented van Trees inequality to the problem of pointwise estimation of a function in the H\"older class with exponent $\beta\in(0,2]$ and constant $L>0$,
\begin{equation*}
    \cH(\beta,L)
    :=
    \begin{cases}
        \bigl\{
        f:\R^d\to\R,
        \;\;
        |f(x)-f(x_0)|\leq L\|x-x_0\|_2^\beta
        \bigr\}
        &\;\,\text{if $\beta\in(0,1]$},
        \\
        \bigl\{
        f:\R^d\to\R \mathrm{\;\; once\;differentiable},
        \;\;
        \|\nabla f(x)-\nabla f(x_0)\|_2 \leq L\|x-x_0\|_2^{\beta-1}
        \bigr\}
        &\;\,\text{if $\beta\in(1,2]$},
    \end{cases}
\end{equation*}
for a general dimension $d\in\N$. 
Suppose we have access to $n$ i.i.d.~pairs $(X_1,Y_1),\ldots,(X_n,Y_n)\in\R^d\times\R$ with
\begin{equation}\label{eq:Y=f(X)+e}
    Y_i=f(X_i)+\varepsilon_i,
\end{equation}
where $\varepsilon_i\given X_i\iid N(0,\sigma^2)$, regression function $f\in\cH(\beta,L)$, and $X_i\sim P_X$ where $P_X$ is absolutely continuous with respect to Lebesgue measure with Radon-Nikodym derivative that is absolutely continuous.   
We will construct a lower bound on the minimax squared risk on $(\cH(\beta,L),\delta)$ for the semi-norm $\delta(f_1,f_2)=|f_1(x_0)-f_2(x_0)|$ for some $x_0\in\R^d$, i.e.
\begin{equation*}
    \inf_{\hat{f}}\sup_{f\in\cH(\beta,L)}\E_f\Bigl[\bigl(\hat{f}(x_0)-f(x_0)\bigr)^2\Bigr],
\end{equation*}
where the infimum is taken over all Borel-measurable functions of the data. 
The augmented van Trees inequality (in the AVT2 form~\eqref{eq:AVT}) may be used to obtain the following near-sharp characterisation on the minimax risk. 

\begin{theorem}[Pointwise H\"older function estimation]\label{thm:all-minimax}
    Consider the model~\eqref{eq:Y=f(X)+e}. For any $\beta\in(0,2]$, $d\in\N$, $L>0$, $\sigma^2>0$, $x_0\in\R^d$,
    \begin{multline*}
        \frac{1}{1.69}\leq 
        \liminf_{n\to\infty}\frac{\inf\limits_{\hat{f}}\sup\limits_{f\in\cH(\beta,L)}\E_f\Bigl[\bigl(\hat{f}(x_0)-f(x_0)\bigr)^2\Bigr]}{\bigl(\frac{d^d (\beta+d)^{2\beta}\Gamma^{2\beta}(1+d/2)}{\pi^{\beta d}\beta^{2\beta}(2\beta+d)^{d}(1\vee\beta)^{2d}}\bigr)^{1/(2\beta+d)}\bigl(\frac{L^{d/\beta}\sigma^2}{p_X(x_0)n}\bigr)^{2\beta/(2\beta+d)}}
        \\
        \leq
        \limsup_{n\to\infty}\frac{\inf\limits_{\hat{f}}\sup\limits_{f\in\cH(\beta,L)}\E_f\Bigl[\bigl(\hat{f}(x_0)-f(x_0)\bigr)^2\Bigr]}{\bigl(\frac{d^d (\beta+d)^{2\beta}\Gamma^{2\beta}(1+d/2)}{\pi^{\beta d}\beta^{2\beta}(2\beta+d)^{d}(1\vee\beta)^{2d}}\bigr)^{1/(2\beta+d)}\bigl(\frac{L^{d/\beta}\sigma^2}{p_X(x_0)n}\bigr)^{2\beta/(2\beta+d)}}
        \leq 
        1.
    \end{multline*}
\end{theorem}

Fixed design, finite sample, and locally asymptotic minimax bounds can similarly be obtained with this machinery. 
The universal constant of $1.69$ in~Theorem~\ref{thm:all-minimax} allows the above result to hold for every $\beta\in(0,2], d\in\N$ but can be improved in individual cases. For example when $(\beta,d)=(2,1)$, corresponding to a univariate differentiable regression function with Lipschitz derivative, we obtain
    \begin{equation*}
        \frac{1}{1.37}\leq 
        \liminf_{n\to\infty}\frac{\inf\limits_{\hat{f}}\sup\limits_{f\in\cH(2,L)}\E_f\Bigl[\bigl(\hat{f}(x_0)-f(x_0)\bigr)^2\Bigr]}
        {\frac{3^{4/5}}{2^2\cdot 5^{1/5}} \cdot \bigl(\frac{L^{1/2}\sigma^2}{p_X(x_0)n}\bigr)^{4/5}}
        \leq
        \limsup_{n\to\infty}\frac{\inf\limits_{\hat{f}}\sup\limits_{f\in\cH(2,L)}\E_f\Bigl[\bigl(\hat{f}(x_0)-f(x_0)\bigr)^2\Bigr]}
        {\frac{3^{4/5}}{2^2\cdot 5^{1/5}} \cdot \bigl(\frac{L^{1/2}\sigma^2}{p_X(x_0)n}\bigr)^{4/5}}
        \leq 
        1.
    \end{equation*}
To obtain a semblance of why the augmented van Trees improves over classical van Trees, note that the tightness in bounding the worst-case risk by the Bayes risk depends on how much of the prior's mass is concentrated around the hardest to distinguish points along a subpath $(P_t)_{t\in [-1,1]}$; typically at the two extremal points $\pm1$. The augmented van Trees inequality, allowing for $\mu(\pm1)$ to be arbitrary, allows for more of the prior's mass to be concentrated in neighbourhoods of $\pm1$ compared to the classical van Trees inequality which requires $\mu(\pm1)=0$; see Figure~\ref{fig:batman} for an example. 

\begin{figure}[ht]
    \centering
    \includegraphics[width=0.9\linewidth]{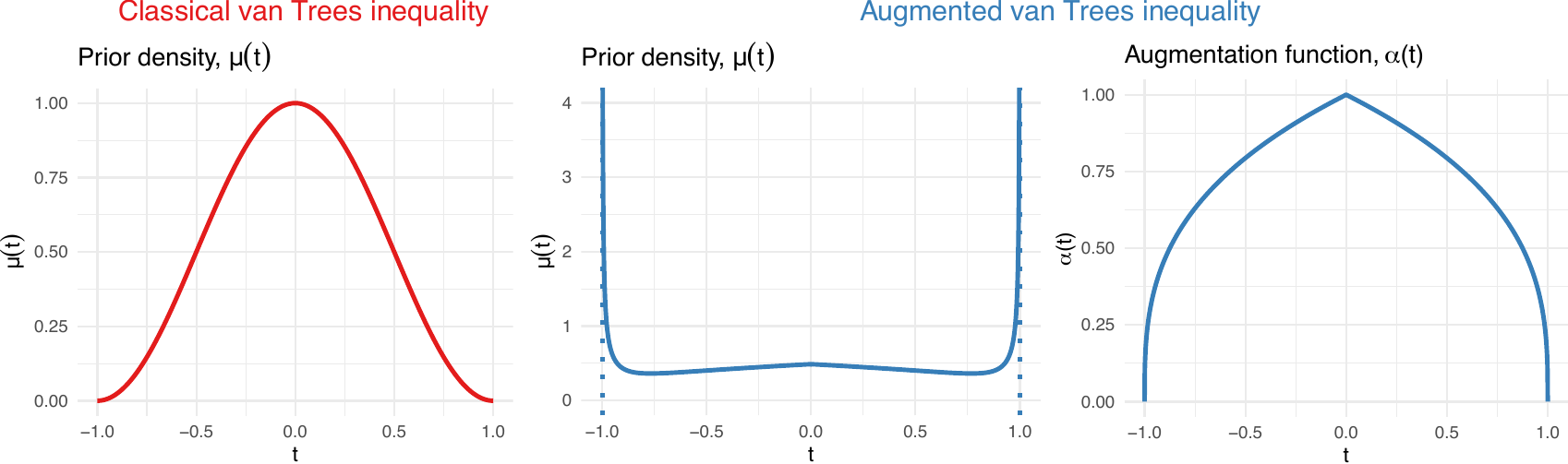}
    \caption{The `optimal' prior densities via the classical van Trees and the augmented van Trees inequalities 
    in the setting of minimax H\"older function estimation in Section~\ref{sec:minimax} with $(\beta,d)=(2,1)$.  
    For the classical van Trees inequality the optimal prior is $\mu(t)=\cos^2({\pi t}/{2})$, with no augmentation. For the augmented van Trees inequality (AVT2) we obtain a lower bound by taking 
    augmentation function $\alpha(t)=(1-|t|)^m$ for $m>0$ 
    and the corresponding optimal prior~\eqref{eq:mu-opt-g}.
    }
    \label{fig:batman}
\end{figure}

Additionally, in the high-dimensional regime $d\to\infty$ we may obtain the exact asymptotic minimax risk for pointwise H\"older function estimation for any $\beta\in(0,2]$, $L>0$, $\sigma>0$, and (smooth) covariate design $p_X$. In the following we denote $\mathcal{B}_0(1)$ to denote the Euclidean ball about $0$ of radius $1$.
\begin{theorem}[Exact risk in high-dimensional regime]\label{thm:asymptotics}
        Suppose $X$ is absolutely continuous with respect to Lebesgue measure, with Radon-Nikodym derivate $p_X$. Further, let $p_X(\cdot)/p_X(x_0)\in\cH(\alpha,L_X)$ for some $\alpha\in(0,1]$, $L_X>0$. 
        Then for any $x_0\in\R^d$ for which $\frac{1}{C}\leq\frac{p_X(x_0)}{1/\mathrm{vol}\mathcal{B}_1(0)}\leq C$ 
        for a constant $C\geq1$ that depends on $x_0$ only, and for any
        $\beta\in(0,2]$, $\sigma^2>0$, 
    \begin{equation}\label{eq:dinfty}
        \lim_{\substack{n,\hspace{0.05em}d\to\infty\\(\log n)/d\to\infty}
        }\frac{\inf_{\hat{f}}\sup_{f\in\cH(\beta,L)}\E_f\Bigl[\bigl(\hat{f}(x_0)-f(x_0)\bigr)^2\Bigr]}{\frac{1}{(2\pi e)^{\beta}(1\vee\beta)^2}d^{\beta}\bigl(\frac{\sigma^2 L^{d/\beta}}{p_X\hspace{-0.1em}(x_0)n}\bigr)^{2\beta/(2\beta+d)}}
        =
        1
    \end{equation}
\end{theorem}
The additional conditions on the covariate density are satisfied for example in the case of $X\sim\mathrm{Unif}\,\mathcal{B}_0(1)$ and $x_0\in\mathrm{int}\,\mathcal{B}_0(1)$ for the constants $(C,\alpha,L_X)=(1,1,1/(1-\|x_0\|_2))$. 
Note the above result does not follow by application of the van Trees inequality, which would only provide an upper bound on~\eqref{eq:dinfty} of $\pi^2$. 

Another benefit of the augmented van Trees inequality is its simplicity of application. Similar minimax lower bounds that quantify explicit constants tend to require rather intricate treatments of specific models; see e.g.~\citet{donoho3, fan-minimax} for explicit constants in specific models. 
In contrast however, the augmented van Trees inequality yields similarly sharp constants whilst being remarkably simple to implement; provided the (asymptotic) Fisher information for a parametric submodel can be quantified a corresponding lower bound immediately follows. This therefore allows us to consider the broad setting here of H\"older function estimation with any smoothness $\beta\in(0,2]$ and any dimension $d\in\N$. These results can therefore also naturally extend to provide lower bounds for arbitrary error distributions with finite Fisher information and sufficiently smooth density.

\section{Augmenting the generalized van Trees inequality}\label{sec:takatsu}

In recent work, \citet{gvt} introduce the generalized van Trees inequality, which extends the classical van Trees bound through the inclusion of an additional absolutely continuous \emph{approximation function} $\phi:T\to\R$. This extension enables the derivation of minimax and local asymptotic minimax lower bounds for non-differentiable functionals as well as for estimation problems arising in irregular statistical models. 
Our flexible augmentation strategy can similarly be incorporated into this generalized framework, again allowing for sharper constants in these irregular settings, as given by the theorem below. 

\begin{theorem}[The augmented generalized van Trees inequality]\label{thm:AGVT}
            Suppose $T\subset\R^d$ is a closed subset, and $(P_t)_{t\in T}$ is a parametric model with $P_t\ll\nu$ for all $t\in T$ for a $\sigma$-finite measure $\nu$ on $(\cX,\mathcal{A})$, and again denote $p(\cdot,t)$ to be the density of $P_t$ with respect to $\nu$, which is absolutely continuous in $t$ for $\nu$-almost all $x$, with Fisher information $\cI(t):=\int_\cX\frac{\partial_t p(x,t)\,\partial_t p(x,t)^\top}{p(x,t)}\,d\nu(x)$ integrable on T, with operator norm $\|\cI(t)\|_{\mathrm{op}}<\infty$. 
            Assume further that:
            \begin{enumerate}[label=(\roman*)]
                \item (Prior densities) The prior density $\mu$ is absolutely continuous on $T$.\label{ass:AGVT-priors}
                \item (Approximation function) The approximation function $\phi:T\to\R^k$ is absolutely continuous with almost-everywhere derivative $\nabla\phi:T\to\R^{k\times d}$. 
                \item (Augmentation function) The augmentation function $\alpha:T\to\R$ is absolutely continuous with $\alpha(t)=0$ for all $t\in\partial T$ on the boundary of $T$.
                \label{ass:AGVT-aug}
            \end{enumerate}
            Then for any absolutely continuous function $\psi:T\to\R^k$ with almost-everywhere derivative $\nabla\psi:T\to\R^{k\times d}$, and for any measurable function $\hat\psi:\cX\to\R^k$, any vector norm $\|\cdot\|:\R^k\to\R_+$, with dual norm $\|\cdot\|_*$,
            \begin{equation}
                \sup_{t\in T}\E_{P_t}\Bigl[\bigl\|\hat\psi({\bf X})-\psi(t)\bigr\|^2\Bigr]
                \geq 
                \sup_{\mu,\alpha}\Biggl(\sup_{\|u\|_*\leq1}\bigl\|\Gamma_{\mu,\phi,\alpha}^{1/2}u\bigr\|
                -
                \biggl(\int_T\|\psi(t)-\phi(t)\|^2 \mu(t)\,dt\biggr)^{1/2}
                \Biggr)_+^2,
            \end{equation}
            where $(\cdot)_+:=\max(\cdot,0)$, and 
            \begin{equation*}
                \Gamma_{\mu,\phi,g}
                :=
                \biggl(\int_T\nabla\phi(t)\alpha(t)\,dt\biggr)^\top
                \biggl(\int_T\frac{\cI(t)\alpha^2(t)+\nabla \alpha(t)\nabla \alpha(t)^\top}{\mu(t)}
                \biggr)^{-1}
                \biggl(\int_T\nabla\phi(t)\alpha(t)\,dt\biggr),
            \end{equation*}
            and 
            where the suprema is taken over all priors $\mu$ satisfying~\ref{ass:AGVT-priors} and all augmentation functions $\alpha$ satisfying~\ref{ass:AGVT-aug}. 
\end{theorem}

As in Theorem~\ref{thm:gvt}, we again do not require the prior density to be `nice' in the senses of~\citet{gassiat, gvt} in that we do not require the density of the prior to vanish on the boundary of $T$. 
In regular estimation settings such as Section~\ref{sec:minimax} the optimal approximation function remains $\phi=\psi$, recovering the standard augmented van Trees inequality of Theorem~\ref{thm:gvt} (see Appendix~\ref{appsec:gvt} for a discussion). 

\section{Discussion}
We introduce a helpful extension of the van Trees inequality, that accommodates priors whose densities don't vanish at the boundaries of its support, and that also provides a tighter lower bound for any fixed prior. Consequently, this augmented van Trees inequality allows one to gain tighter bounds on the worst case risk. 
We apply this inequality to the problem of minimax H\"older function estimation to near-sharp constants in this problem, and the exact constant in the high-dimensional regime. 
We extend the (augmented) van Trees inequality to accommodate loss functions beyond mean squared error, and for estimation in irregular models.

The augmented van Trees inequality gives a remarkably simple and accessible method for deriving nonparametric minimax bounds, that often attains competitively (if not better) constants compared to the more involved convergence of experiments theory. 
In the case of constructing minimax lower bounds on the mean squared error, the problem simply reduces to that of quantifying the Fisher information of a path through the model class. Therefore we believe the augmented van Trees inequality acts as a very helpful tool in the statisticians toolbox for attaining simple yet strong minimax lower bounds in a range of settings.

\paragraph*{Funding.} EHY was supported by European Research Council Advanced Grant 101019498.

\bibliography{avt.bib}
\appendix

\section{Proofs of results}

\subsection{Proof of Theorem~\ref{thm:gvt}}

Note that Theorem~\ref{thm:gvt} follows as a corollary of Theorem~\ref{thm:AGVT}. However, due to the simplicity of the case of Theorem~\ref{thm:gvt} we include a brief proof in this restricted case, which is still sufficient for all our results in Section~\ref{sec:minimax}. 

\begin{proof}[Proof of Theorem~\ref{thm:gvt}]
    Define
    \begin{equation*}
        \zeta(x,t):=\frac{\partial_t\big(p(x,t)\alpha(t)\big)}{p(x,t)\mu(t)},
    \end{equation*}
    where $\partial_t(g(x,t))$ denotes the derivative with respect to $t$ of $g:\cX\times T\to\R$. 
    Then by the Cauchy--Schwarz inequality,
    \begin{align}
        \int_T\E_{P_t}\big[(\hat{t}(&\boldsymbol{X})-t)^2\big]\mu(t)\,dt
        \notag
        \\&\geq
        \bigg(\int_T\int_\cX p(x,t)\mu(t)\zeta^2(x,t)\,d\nu(x)\,dt\bigg)^{-1}
        \bigg(\int_T\int_\cX(\hat{t}(x)-t)p(x,t)\mu(t)\zeta(x,t)\,d\nu(x)\,dt\bigg)^2.
        \label{eq:3}
    \end{align}
    We will simplify the numerator and denominator of~\eqref{eq:3} in turn. 
    First note that as $p(x,t)$, $\mu(t)$ and $\alpha(t)$ are absolutely continuous with $\alpha(t_1)=\alpha(t_2)=0$,
    \begin{equation}\label{eq:1}
        \int_Tp(x,t)\mu(t)\zeta(x,t)\,dt
        =\int_T \partial_t\big(p(x,t)\alpha(t)\big)\,dt=0,
    \end{equation}
    for $\nu$-almost all $x$.  
    Thus
    \begin{align}\label{eq:2}
        \int_T\int_\cX \hat{t}(x)p(x,t)\mu(t)\zeta(x,t)\,d\nu(x)\,dt=0.
    \end{align}
    By similar arguments (also by absolute continuity) and by integration by parts,
    \begin{equation*}
        \int_Tt\,\partial_t\big(p(x,t)\alpha(t)\big)\,dt = -\int_Tp(x,t)\alpha(t)\,dt,
    \end{equation*}
    thus
    \begin{equation}\label{eq:4}
        \int_T\int_\cX tp(x,t)\mu(t)\zeta(x,t)\,d\nu(x)\,dt
        =
        -\int_T\int_\cX p(x,t) \,d\nu(x) \alpha(t)\,dt
        =
        -\int_T \alpha(t) \,dt
        .
    \end{equation}
    The numerator term of~\eqref{eq:3} therefore simplifies to
    \begin{align}\label{eq:5}
        \biggl(\int_T\int_\cX (\hat{t}(x)-t)p(x,t)\mu(t)\zeta(x,t)\,d\nu(x)\,dt\biggr)^2
        =
        \biggl(\int_T \alpha(t)\,dt\biggr)^2
    \end{align}
    by applying~\eqref{eq:2}~and~\eqref{eq:4}.

For the denominator term of~\eqref{eq:3},
\begin{align}
    &\quad\;
    \int_T\int_\cX p(x,t)\mu(t)\zeta^2(x,t)\,d\nu(x)\,dt
    \notag
    \\
    &=
    \int_T\int_\cX \frac{\big(\partial_tp(x,t)\alpha(t)+p(x,t)\alpha'(t)\big)^2}{p(x,t)\mu(t)}\,d\nu(x)\,dt
    \notag
    \\
    &=
    \int_T\int_\cX \frac{(\partial_tp(x,t))^2}{p(x,t)}\frac{\alpha^2(t)}{\mu(t)}\,d\nu(x)\,dt
    +
    2\int_T\int_\cX \partial_tp(x,t)\,d\nu(x)\frac{\alpha(t)\alpha'(t)}{\mu(t)}\,dt
    +
    \int_T\int_\cX p(x,t)\frac{\big(\alpha'(t)\big)^2}{\mu(t)}\,d\nu(x)\,dt
    \notag
    \\
    &=
    \int_T\cI(t)\frac{\alpha^2(t)}{\mu(t)}\,dt
    +
    \int_T\frac{\alpha'(t)^2}{\mu(t)}\,dt,
    \label{eq:6}
\end{align}
because the cross term is zero; note firstly that by Jensen's inequality
\begin{equation*}
    \int_\cX \bigl|\partial_t p(x,t)\bigr|\,d\nu(x)
    \leq 
    \sqrt{\cI(t)}<\infty,
\end{equation*}
for all $t\in T$, thus we may apply the dominated convergence theorem to obtain
\begin{equation*}
    \int_\cX \partial_tp(x,t)\,d\nu(x)
    =
    \frac{\partial}{\partial t}\int_\cX p(x,t)\,d\nu(x)=0.
\end{equation*}
Combining~\eqref{eq:3},~\eqref{eq:5}~and~\eqref{eq:6} we obtain the required result;
\begin{equation*}
    \int_T\E_{P_t}\big[(\hat{t}(\boldsymbol{X})-t)^2\big]\mu(t)\,dt
    \geq
    \frac{\big(\int_T\alpha(t)\,dt\big)^2}{\int_T\frac{\cI(t)\alpha^2(t)+\alpha'(t)^2}{\mu(t)}\,dt}.
\end{equation*}

\end{proof}

\subsection{Proof of Theorem~\ref{prop:gvt-simple}}
\begin{proof}[Proof of Theorem~\ref{prop:gvt-simple}]
    Fix some $\alpha\in\mathcal{A}$. By the Cauchy--Schwarz inequality,
    \begin{equation*}
        \biggl(\int_T\frac{\cI \alpha^2+(\alpha')^2}{\mu}\biggr)
        =
        \biggl(\int_T\mu\biggr)
        \biggl(\int_T\frac{\cI \alpha^2+(\alpha')^2}{\mu}\biggr)
        \geq
        \biggl(\int_T\sqrt{\cI \alpha^2+(\alpha')^2}\biggr)^2,
    \end{equation*}
    with equality if and only if
    \begin{equation*}
        \mu \propto \sqrt{\cI \alpha^2+(\alpha')^2}
    \end{equation*}
    As $\mu$ is a density it follows that the optimal prior $\mu^*_{\alpha}$ is
    \begin{equation}\label{eq:mu-opt}
        \mu^*_{\alpha} := \frac{\sqrt{\cI \alpha^2+(\alpha')^2}}{\int_T\sqrt{\cI \alpha^2+(\alpha')^2}},
    \end{equation}
    which is absolutely continuous as all of $\cI,\alpha,\alpha'$ are absolutely continuous by assumption. As $\alpha$ is differentiable with absolutely continuous derivative $\max\bigl\{\int_T|\alpha|,\int_T|\alpha'|\bigr\}<\infty$. 
    Applying Theorem~\ref{thm:gvt} with $\mu=\mu_\alpha^*$,
    \begin{equation*}
        \sup_{\mu}\frac{\bigl(\int_T\alpha \bigr)^2}{\int_T\frac{\cI \alpha^2+(\alpha')^2}{\mu}}
        =
        \frac{\bigl(\int_T\alpha \bigr)^2}{\int_T\frac{\cI \alpha^2+(\alpha')^2}{\mu^*_{\alpha}}}
        =
        \Biggl(\frac{\int_T\alpha }{\int_T\sqrt{\cI \alpha^2+(\alpha')^2}}\Biggr)^2.
    \end{equation*}
    The result follows by taking the supremum over all $\alpha\in\mathcal{A}$.
\end{proof}

\subsection{Proof of the result of Example~\ref{ex:aVT1}}

\begin{proof}[Proof of Example~\ref{ex:aVT1}]
    Fix $\delta\in(0,1)$. Consider the augmentation function $\alpha:[-1,1]\to\R$ given by
    \begin{equation*}
        \alpha(t) := 
        \begin{cases}
            e^{-k|t|} &\quad\text{if }|t|\leq 1-\delta,
            \\
            
            \tfrac{1}{\delta}e^{-(1-\delta)k}(1-|t|)&\quad\text{if } 1-\delta< |t|\leq 1.   
        \end{cases}
        \end{equation*}
        By construction $\alpha$ is absolutely continuous with $\alpha(1)=\alpha(-1)=0$. 
        Also define $f:[0,1]\to\R$ given by $$f(t):=\frac{\alpha(t)}{-\alpha'(t)}
        =\begin{cases}
            \frac{1}{k}&\quad\text{if }t\in[0,1-\delta],
            \\
            \tfrac{1-t}{\delta}&\quad\text{if }t\in(1-\delta,1].
        \end{cases}
        $$
        As both $\alpha$ restricted to the domain $[0,1]$ and $f$ is monotonic, we may write
        \begin{equation*}
            f(\alpha(t))
            =
             \begin{cases}
                 \delta e^{(1-\delta)k}\alpha(t)
                 &\text{if }\alpha(t)\in\bigl[0,e^{-(1-\delta)k}\bigr],
                 \\
                 \tfrac{1}{k}
                 &\text{if }\alpha(t)\in\bigl(e^{-(1-\delta)k},1\bigr].
             \end{cases}
        \end{equation*}
        Then
        \begin{multline*}
            \frac{\int_{-1}^1\alpha(t)\,dt}{\int_{-1}^1\sqrt{\cI\alpha^2(t)+\alpha'(t)^2}\,dt}
            =
            \frac{\int_0^1\alpha(t)\,dt}{\int_0^1\sqrt{\cI\alpha^2(t)+\alpha'(t)^2}\,dt}
            =
            \frac{\int_{0}^1f(\alpha)\,d\alpha}{\int_{0}^1\sqrt{\cI(f(\alpha))^2+1}\,d\alpha}
            \\
            =
            \frac{\frac{\delta}{2}e^{-(1-\delta)k} + \tfrac{1}{k}\bigl(1-e^{-(1-\delta)k}\bigr)}{\sqrt{\tfrac{\cI}{k^2}+1}\,\bigl(1-e^{-(1-\delta)k}\bigr)+e^{-(1-\delta)k}\int_0^1\sqrt{\delta^2\cI u^{2}+1}\,du}.
        \end{multline*}
        Taking $\delta\searrow0$ we obtain
        \begin{multline*}
        \sup_{\alpha}\frac{\int_{-1}^1\alpha(t)\,dt}{\int_{-1}^1\sqrt{\cI\alpha^2(t)+\alpha'(t)^2}\,dt}
        \geq
        \sup_{k>0}
            \lim_{\delta\searrow0}\frac{\int_{-1}^1\alpha_\delta(t)\,dt}{\int_{-1}^1\sqrt{\cI\alpha_\delta^2(t)+\alpha_\delta'(t)^2}\,dt} 
            \\
            =
            \sup_{k>0}
            \frac{1}{\sqrt{\cI+k^2}+\frac{k}{e^k-1}}
            \geq
            \lim_{k\searrow0}
            \frac{1}{\sqrt{\cI+k^2}+\frac{k}{e^k-1}}
            =
            \frac{1}{\sqrt{\cI}+1}.
        \end{multline*}
\end{proof}

\subsection{Proof of the result of Example~\ref{ex:hypergeo}}\label{appsec:hypergeo}

\begin{proof}[Proof of the result of Example~\ref{ex:hypergeo}]
First note $\alpha(1)=\alpha(-1)=0$ and $\alpha$ is absolutely continuous. Further, $\int_{-1}^1|\alpha|=2\int_0^1 (1-t)^m\,dt=\frac{2}{m+1}\leq 2$ and $\int_{-1}^1|\alpha'|=-2\int_0^1\alpha'=2(\alpha(0)-\alpha(1))=2$. The class $\mathcal{A}$ therefore satisfies the assumptions of Theorem~\ref{prop:gvt-simple}. Then direct calculations show that $\int_{-1}^1\alpha =\frac{2}{m+1}$ and
\begin{equation*}
    \int_{-1}^1\sqrt{\cI \alpha^2+(\alpha')^2}
    =
    2\int_0^1\sqrt{\cI t^{2m}+m^2t^{2(m-1)}}\,dt
    =2\int_0^1 t^{m-1}\sqrt{\cI t^2+m^2} \,dt
    =
    2 _2F_1\bigl(\tfrac{-1}{2},\tfrac{m}{2},\tfrac{m}{2}+1;\tfrac{-\cI}{m^2}\bigr),
\end{equation*}
which combined with Theorem~\ref{prop:gvt-simple} gives the result.
\end{proof}

\subsection{Proof of Theorem~\ref{thm:genrealized-loss}}

\begin{proof}[Proof of Theorem~\ref{thm:genrealized-loss}]
    We adopt the same notation as the proof of Theorem~\ref{thm:gvt}. 
    Then by H\"older's inequality
    \begin{multline*}
        \biggl\{\int_T\E_{P_t}\bigl[|\hat{t}(\boldsymbol{X})-t|^p\bigr]\mu(t)\,dt\biggr\}^{1/p}
        \\
        \geq
        \biggl\{\int_T\int_\cX p(x,t)\mu(t)|\zeta(x,t)|^q\,d\nu(x)\,dt\biggr\}^{-1/q}
        \biggl|\int_T\int_\cX(\hat{t}(x)-t)p(x,t)\mu(t)\zeta(x,t)\,d\nu(x)\,dt\biggr|.
    \end{multline*}
    Then
    \begin{equation*}
        |\zeta(x,t)|^q
        =
        \frac{1}{\mu(t)^q}\biggl|\frac{\partial_tp(x,t)}{p(x,t)}\alpha(t) + \alpha'(t)\biggr|^q,
    \end{equation*}
    and so
    \begin{equation*}
        \int_T\int_\cX p(x,t)\mu(t)|\zeta(x,t)|^q\,d\nu(x)\,dt
        =
        \int_T \mu(t)^{1-q}\,\E_{P_t}\bigl[|\alpha'(t)+\rho_t(\boldsymbol{X})\alpha(t)|^q\bigr]\,dt.
    \end{equation*}
    Further, take
    \begin{equation*}
        \mu(t) := \frac{\bigl\{\E_{P_t}\bigl[|\alpha'(t)+\rho_t(\boldsymbol{X})\alpha(t)|^q\bigr]\bigr\}^{1/q}}{\int_T\bigl\{\E_{P_{\tau}}\bigl[|\alpha'(\tau)+\rho_{\tau}(\boldsymbol{X})\alpha(\tau)|^q\bigr]\bigr\}^{1/q}d\tau}.
    \end{equation*}
    Alongside equation~\eqref{eq:5}, we conclude
    \begin{equation*}
        \sup_{t\in T}\bigl\{\E_{P_t}\bigl[|\hat{t}(\boldsymbol{X})-t|^p\bigr]\bigr\}^{1/p}
        \geq
        \frac{\bigl|\int_T\alpha(t)\,dt\bigr|}{\int_T\bigl\{\E_{P_t}\bigl[|\alpha'(t)+\rho_t(\boldsymbol{X})\alpha(t)|^q\bigr]\bigr\}^{1/q}dt}
    \end{equation*}
\end{proof}

\subsection{Proof of Theorem~\ref{thm:all-minimax}}

We divide the proof of Theorem~\ref{thm:all-minimax} into the lower and upper bounds individually.

\begin{lemma}\label{lem:LB}
As in the setup of Theorem~\ref{thm:all-minimax},
    \begin{equation*} 
        \liminf_{n\to\infty}\frac{\inf_{\hat{f}}\sup_{f\in\cH(\beta,L)}\E_f\Bigl[\bigl(\hat{f}(x_0)-f(x_0)\bigr)^2\Bigr]}{\bigl(\frac{d^d (\beta+d)^{2\beta}\Gamma^{2\beta}(1+d/2)}{\pi^{\beta d}\beta^{2\beta}(2\beta+d)^{d}(1\vee\beta)^{2d}}\bigr)^{1/(2\beta+d)}\bigl(\frac{L^{d/\beta}\sigma^2}{p_X(x_0)n}\bigr)^{2\beta/(2\beta+d)}}
        \geq
        A_{\beta,d},
    \end{equation*}
    for the function
    \[A_{\beta,d}:=A\Bigl(\frac{2\beta}{2\beta+d}\Bigr),
    \qquad
    A(a) := 
    \frac{1}{a^a(1-a)^{1-a}}\sup_{\substack{\lambda>0,\\m>0}}\frac{\lambda^a}{\bigl\{(m+1)\hspace{0em}_2F_1\bigl(-\frac{1}{2},\frac{m}{2},\frac{m}{2}+1;-\frac{\lambda}{m}\bigr)\bigr\}^2}.
    \]
    Moreover,
    \begin{equation*}
        \inf_{\beta\in(0,2]}\inf_{d\in\N}A_{\beta,d}\geq\frac{1}{1.69},
        \quad\text{and}\quad
        A_{2,1}\geq\frac{1}{1.37},
    \end{equation*}
\end{lemma}

\begin{proof}[Proof of Lemma~\ref{lem:LB}]
Consider the family of distributions $(P_t)_{t\in[-1,1]}$ given by
\begin{gather*}
    Y_i\given X_i\iid N\bigl(f_t(X_i),\,\sigma^2\bigr),
    \qquad
    f_t(x) := \frac{t L h_n^\beta }{1\vee\beta}K\Bigl(\frac{x-x_0}{h_n}\Bigr),
    \label{eq:f_t}
    \\
    K(u)=\bigl(1-\|u\|_2^\beta\bigr)_+\,
    ,
    \qquad
    R(K) := \int_\R K^2(u)\,du,
    \qquad h_n := \biggl(\frac{(1\vee\beta)^2 \sigma^2 \lambda}{L^2p_X(x_0)R(K)n}\biggr)^{1/(2\beta+d)},
\end{gather*}
where $(\cdot)_+:=\max(\cdot\,,0)$. For $\beta\in(0,1]$,
\begin{equation*}
    \frac{|f_t(x)-f_t(x_0)|}{\|x-x_0\|_2^\beta}
    =
    \frac{tLh_n^\beta\bigl|K\bigl(\frac{x-x_0}{h_n}\bigr)-K(0)\bigr|}{\|x-x_0\|_2^\beta}
    \leq L,
\end{equation*}
and for $\beta\in(1,2]$,
\begin{equation*}
    \frac{\|\nabla f_t(x)-\nabla f_t(x_0)\|_2}{\|x-x_0\|_2^{\beta-1}}
    =
    \frac{tLh_n^{\beta-1}\bigl\|\nabla K\bigl(\frac{x-x_0}{h_n}\bigr)-\nabla K(0)\bigr\|_2}{\beta\|x-x_0\|_2^{\beta-1}}
    \leq L.
\end{equation*}
Therefore $f_t\in\cH(\beta,L)$. 
The Fisher information is
\begin{equation*}
    \cI(t) = \frac{L^2 n h_n^{2\beta+d}}{(1\vee\beta)^2\sigma^2} \int K^2(u)p_X(x_0+h_nu)\,du
    \to 
    \frac{L^2nh_n^{2\beta+d}R(K)p_X(x_0)}{(1\vee\beta)^2\sigma^2}=\lambda
    ,
\end{equation*}
as $n\to\infty$. 
For an arbitrary Borel measurable estimator $\hat{f}_n(x_0)$ define~$\hat{t}_n:=\frac{1\vee\beta}{Lh_n^{\beta}}\hat{f}_n(x_0)$. Then by the augmented van Trees 2 inequality~\eqref{eq:AVT} (see Example~\ref{ex:hypergeo})
\begin{align*}
    &\quad\liminf_{n\to\infty}\sup_{f\in\cH(\beta,L)}n^{2\beta/(2\beta+d)}\E_{f}\big[(\hat{f}_n(x_0)-f(x_0))^2\big]
    \\
    &\geq
    \liminf_{n\to\infty}\sup_{t\in[-1,1]}n^{2\beta/(2\beta+d)}\E_{P_t}\big[(\hat{f}_n(x_0)-f_t(x_0))^2\big]
    \\
    &\geq
    \liminf_{n\to\infty}\frac{L^2 h_n^{2\beta}n^{2\beta/(2\beta+d)}}{(1\vee\beta)^2}
    \sup_{t\in[-1,1]}\E_{P_t}\big[(\hat{t}_n-t)^2\big]
    \\
    &\geq
    \biggl(\frac{L^{d/\beta}\sigma^2}{(1\vee\beta)^{d/\beta} p_X(x_0) R(K)}\biggr)^{2\beta/(2\beta+d)}
    \sup_{\substack{\lambda>0,\\m>0}}\biggl(\frac{\lambda^{\beta/(2\beta+d)}}{(m+1)\hspace{0em}_2F_1\bigl(-\frac{1}{2},\frac{m}{2},\frac{m}{2}+1;-\frac{\lambda}{m}\bigr)}\biggr)^2,
\end{align*}
as $n\to\infty$. Now, multivariable calculus calculations gives
\begin{equation}\label{eq:R2}
    R(K)=\frac{2\pi^{d/2}\beta^2}{(\beta+d)(2\beta+d)\Gamma(1+\frac{d}{2})}.
\end{equation}
Therefore
\begin{multline*}
    \liminf_{n\to\infty}\sup_{f\in\cH(\beta,L)}n^{2\beta/(2\beta+d)}\E_{f}\big[(\hat{f}_n(x_0)-f(x_0))^2\big]
    \\
    \geq
    \biggl(\frac{d^d (\beta+d)^{2\beta}\Gamma^{2\beta}(1+d/2)}{\pi^{\beta d}\beta^{2\beta}(2\beta+d)^{d}(1\vee\beta)^{2d}}\biggr)^{1/(2\beta+d)}
    \biggl(\frac{L^{d/\beta}\sigma^2}{p_X(x_0) }\biggr)^{2\beta/(2\beta+1)}
    A\biggl(\frac{2\beta}{2\beta+d}\biggr).
\end{multline*}
Numerical computation yields $\inf_{a\in[0,4/5]}A(a)\geq\frac{1}{1.69}$, with the special case of $(\beta,d)=(2,1)$ obtained by evaluating $A(4/5)\geq\frac{1}{1.37}$. 

\end{proof}

\begin{lemma}\label{lem:UBB}
    As in the setup of Theorem~\ref{thm:all-minimax},
    \begin{equation*} 
        \limsup_{n\to\infty}\frac{\inf_{\hat{f}}\sup_{f\in\cH(\beta,L)}\E_f\Bigl[\bigl(\hat{f}(x_0)-f(x_0)\bigr)^2\Bigr]}{\bigl(\frac{d^d (\beta+d)^{2\beta}\Gamma^{2\beta}(1+d/2)}{\pi^{\beta d}\beta^{2\beta}(2\beta+d)^{d}(1\vee\beta)^{2d}}\bigr)^{1/(2\beta+d)}\bigl(\frac{L^{d/\beta}\sigma^2}{p_X(x_0)n}\bigr)^{2\beta/(2\beta+d)}}
        \leq
        1.
    \end{equation*}
\end{lemma}

\begin{proof}

We consider the local constant (Nadaraya--Watson) estimator
\begin{equation*}
    \hat{f}_{n,h}(x_0) := \frac{\frac{1}{n}\sum_{i=1}^nK_h(X_i-x_0)Y_i}{\frac{1}{n}\sum_{i=1}^nK_h(X_i-x_0)+n^{-r}},
\end{equation*}
with bandwidth an kernel
\begin{equation}\label{eq:h-K}
    h := \biggl(\frac{d (1\vee\beta)^2\sigma^2(x_0) R(K)}{2\beta L^2\mu_\beta^2(K)p_X(x_0)n}\biggr)^{1/(2\beta+d)},
    \qquad
    K(u) := \frac{(\beta+d)\Gamma(1+d/2)}{\pi^{d/2}\beta}\bigl(1-\|u\|_2^\beta\bigr)_+,
\end{equation}
and where $R(K):=\int_\R K^2(u)\,du$, $\mu_\beta(K):=\int_\R K(u)\|u\|_2^\beta \,du$, $K_h(\cdot):=K(\cdot/h)/h^d$, and $r>0$ is an arbitrarily large positive constant. Note the additional $n^{-r}$ term in the denominator enforces the Nadaraya-Watson estimator is well defined on the event $\frac{1}{n}\sum_{i=1}^nK_h(X_i-x_0)=0$. 

We proceed to show that
\begin{equation*}
    \lim_{n\to\infty}\sup_{f\in\cH(\beta,L)}\E_f\Bigl[\bigl(\hat{f}_{n,h}(x_0)-f(x_0)\bigr)^2\Bigr]
    \leq
    \biggl(\frac{d^d (\beta+d)^{2\beta}\Gamma^{2\beta}(1+d/2)}{\pi^{\beta d}\beta^{2\beta}(2\beta+d)^{d}(1\vee\beta)^{2d}}\biggr)^{1/(2\beta+d)}\biggl(\frac{L^{d/\beta}\sigma^2}{p_X(x_0)n}\biggr)^{2\beta/(2\beta+d)}
    .
\end{equation*}

Define $e(x):=f(x)-f(x_0)$ and $\sigma^2(x):=\E(\varepsilon^2\given X=x)$. Then
\begin{equation*}
    \E_f\bigl[(\hat{f}_{n,h}(x_0)-f(x_0))^2\bigr]
    =
    \E\Biggl[\biggl(\frac{\frac{1}{n}\sum_{i=1}^nK_h(X_i-x_0)e(X_i)}{\frac{1}{n}\sum_{i=1}^nK_h(X_i-x_0)+n^{-r}}\biggr)^2\,\Biggr]
    +
    \frac{1}{n}\E\Biggl[\frac{\frac{1}{n}\sum_{i=1}^nK_h^2(X_i-x_0)\sigma^2(X_i)}{\bigl(\frac{1}{n}\sum_{i=1}^nK_h(X_i-x_0)+n^{-r}\bigr)^2}\Biggr].
\end{equation*}
Define $D_n:=\frac{1}{n\,p_X(x_0)}\sum_{i=1}^nK_h(X_i-x_0)+n^{-r}$ and $D:=\frac{1}{p_X(x_0)}\E[K_h(X-x_0)]+n^{-r}=\E(D_n)$.  Then
\begin{align*}
    \frac{1}{n}\E\biggl[\frac{\frac{1}{n}\sum_{i=1}^nK_h^2(X_i-x_0)\sigma^2(X_i)}{D_n^2}\biggr]
    &=
    \frac{1}{nD^2}\E\biggl[\Bigl(\frac{1}{n}\sum_{i=1}^nK_h^2(X_i-x_0)\sigma^2(X_i)\Bigr)\biggr]
    \\
    &\qquad +
    \frac{1}{n}\E\biggl[\biggl(\frac{1}{D_n^2}-\frac{1}{D^2}\biggr)\Bigl(\frac{1}{n}\sum_{i=1}^nK_h^2(X_i-x_0)\sigma^2(X_i)\Bigr)\biggr]
    \\
    &=
    \frac{1}{nD^2}\,\E\bigl[K_h^2(X-x_0)\sigma^2(X)\bigr]
    \\
    &\qquad
    -
    \frac{1}{D^2}\E\biggl[\biggl(\frac{D_n^2-D^2}{D_n^2}\biggr)\Bigl(\frac{1}{n^2}\sum_{i=1}^nK_h^2(X_i-x_0)\sigma^2(X_i)\Bigr)\biggr].
\end{align*}
For the second term,
\begin{align*}
    \E\biggl[\biggl|\biggl(\frac{D_n^2-D^2}{D^2D_n^2}\biggr)\Bigl(&\frac{1}{n^2}\sum_{i=1}^nK_h^2(X_i-x_0)\sigma^2(X_i)\Bigr)\biggr|\biggr]
    \leq
    \frac{1}{D^2}
    \E\biggl[\biggl(\frac{1}{n^2}\sum_{i=1}^nK_h^2(X_i-x_0)\sigma^2(X_i)\biggr)^2\,\biggr]^{1/2}
    \\
    &\leq
    \frac{\sigma_\infty^2}{D^2}
    \E\biggl[\biggl(\frac{D_n^2-D^2}{D_n^2}\biggr)^2\biggr]^{1/2}
    \biggl\{\frac{n-1}{n^3}\bigl(\E[K_h^2(X-x_0)]\bigr)^2+\frac{1}{n^3}\E[K_h^4(X-x_0)]\biggr\}^{1/2}
    \\
    &\leq
    \E\biggl[\biggl(\frac{D_n^2-D^2}{D_n^2}\biggr)^2\biggr]^{1/2}
    O\biggl(\frac{1}{nh^d}\biggr).
\end{align*}
Finally, define $E_p:=\E\bigl[|D_n-D|^p\bigr]$. 
Then, for any $q>0$,
\begin{align*}
    \E\biggl[\biggl(\frac{D_n^2-D^2}{D_n^2}\biggr)^2\biggr]
    &\leq
    n^{4r}\E\bigl[(D_n^2-D^2)^2\ind_{(|D_n-D|>D/2)}\bigr]
    +
    \frac{2^4}{D^4}\E\bigl[(D_n^2-D^2)^2\ind_{(|D_n-D|\leq D/2)}\bigr]
    \\
    &\leq
    \frac{2^q n^{4r}}{D^q}\E\bigl[(D_n+D)^2|D_n-D|^{q+2}\bigr]
    +
    \frac{2^4}{D^4}\E\bigl[(D_n+D)^2(D_n-D)^2\bigr]
    \\
    &\leq
    \frac{2^{q+1}}{D^q} n^{4r}E_{q+4}
    +\frac{2^{q+3}}{D^{q-2}} n^{4r}E_{q+2}
    + \frac{2^5}{D^4}E_4 + \frac{2^7}{D^2}E_2,
\end{align*}
with the final inequality following because $(x+y)^2\leq 2(x-y)^2+8y^2$ for all $x,y\in\R$. 
By Rosenthal's inequality, for any $p\geq 2$, there exists a constant $C_p$ that depends only on $p$ such that
\begin{equation*}
    E_p \leq C_p \biggl(\frac{\E(|K_h(X-x_0)|^p)}{n^{p-1}} + \frac{\{\Var(K_h(X-x_0))\}^{p/2}}{n^{p/2}}\biggr)
    = O\biggl(\frac{1}{(nh^d)^{p/2}}\biggr).
\end{equation*}
Therefore, for any $q>0$, 
\begin{equation*}
    \E\biggl[\biggl(\frac{D_n^2-D^2}{D_n^2}\biggr)^2\biggr]
    =
    O\biggl(\frac{1}{nh^d}+\frac{n^{4r}}{(nh^d)^{1+q/2}}\biggr) 
    =
    O\biggl(n^{-2\beta/(2\beta+d)} + n^{4r-\beta(q+2)/(2\beta+d)}\biggr).
\end{equation*}
Therefore taking $q>4(2+d/\beta)r-2$ we obtain
\begin{equation*}
    \E\biggl[\biggl(\frac{D_n^2-D^2}{D_n^2}\biggr)^2\biggr] = o(1),
\end{equation*}
thus
\begin{equation*}
    \E\biggl[\biggl(\frac{D_n^2-D^2}{D^2D_n^2}\biggr)\Bigl(\frac{1}{n^2}\sum_{i=1}^nK_h^2(X_i-x_0)\sigma^2(X_i)\Bigr)\biggr]=o\biggl(\frac{1}{nh^d}\biggr).
\end{equation*}
By analogous arguments,
\begin{equation*}
    \E\Biggl[\frac{\bigl(\frac{1}{n}\sum_{i=1}^nK_h(X_i-x_0)e(X_i)\bigr)^2}{D_n^2}\,\Biggr]
    =
   o\bigl(h^{2\beta}\bigr).
\end{equation*}
Moreover,
\begin{align*}
    \E\biggl[\biggl(\frac{1}{n}\sum_{i=1}^nK_h(X_i-x_0)e(X_i)\biggr)^2\,\biggr]
    &=
    \frac{n-1}{n}\bigl\{\E\bigl[K_h(X-x_0)e(X)\bigr]\bigr\}^2
    +
    \frac{1}{n}\E\bigl[K_h^2(X-x_0)e^2(X)\bigr]
    \\
    &=
    \bigl\{\E\bigl[K_h(X-x_0)e(X)\bigr]\bigr\}^2 + o\biggl(h^{2\beta}+\frac{1}{nh^d}\biggr).
\end{align*}
Therefore
\begin{equation}\label{eq:step1}
    \E_f\bigl[(\hat{f}_{n,h}(x_0)-f(x_0))^2\bigr]
    =
    \frac{1}{nD^2}\E\bigl[K_h^2(X-x_0)\sigma^2(X)\bigr] + \frac{1}{D^2}\bigl\{\E\bigl[K_h(X-x_0)e(X)\bigr]\bigr\}^2 + o\biggl(h^{2\beta}+\frac{1}{nh^d}\biggr).
\end{equation}
Now, by absolute continuity of $p_X$,
\begin{equation*}
    D = \frac{1}{p_X(x_0)}\E[K_h(X-x_0)] = \frac{1}{p_X(x_0)}\int_\R K(u)p_X(x_0+hu)\,du
    = 1+o(1).
\end{equation*}
Additionally with absolute continuity of $\sigma^2$,
\begin{equation*}
    \E\bigl[K_h^2(X-x_0)\sigma^2(X)\bigr]
    =
    \frac{1}{h^d}\int_\R K^2(u)\sigma^2(x_0+hu)p_X(x_0+hu)\,du
    =
    \frac{R(K)p_X(x_0)\sigma^2(x_0)}{h^d}(1+o(1)).
\end{equation*}
When $\beta\in(0,1]$,
\begin{align*}
    \bigl|\E\bigl[K_h(X-x_0)e(X)\bigr]\bigr|
    \leq
    &\int_{\R^d} K(u) \bigl|f(x_0+hu)-f(x_0)\bigr|p_X(x_0+hu)du
    \\
    &\leq
    Lh^\beta\int_{\R^d} K(u)\|u\|_2^\beta p_X(x_0+hu)\,du
    = p_X(x_0)Lh^\beta\mu_\beta(K)(1+o(1)).
\end{align*}
On the other hand, when $\beta\in(1,2]$,
\begin{align*}
    \bigl|\E\bigl[K_h(X-x_0)&e(X)\bigr]\bigr|
    =
    \biggl|\int_{\R^d} K(u) \bigl(f(x_0+hu)-f(x_0)\bigr)p_X(x_0+hu)\,du\biggr|
    \\
    &= h\biggl|\int_{\R^d}\int_0^1 K(u) u^\top\bigl(\nabla f(x_0+thu)-\nabla f(x_0)\bigr)\,dt\,p_X(x_0+hu)\,du\biggr|
    \\
    &\leq
    h\int_{\R^d}\int_0^1 t\, K(u) \|u\|_2 \|\nabla f(x_0+thu)-\nabla f(x_0)\|_2\,dt\,p_X(x_0+hu)\,du
    \\
    &\leq
    Lh^{\beta} \int_0^1 t^{\beta-1} dt \int_{\R^d} K(u) \|u\|_2^\beta p_X(x_0+hu) \,du
    = \frac{p_X(x_0)Lh^{\beta}\mu_\beta(K)}{\beta}(1+o(1)) ,
\end{align*}
obtained by taking a Taylor expansion with integral form of remainder. 
Combining~\eqref{eq:step1} with the above,
\begin{equation*}
    \E_f\bigl[(\hat{f}_{n,h}(x_0)-f(x_0))^2\bigr]
    =
    \frac{L^2h^{2\beta}\mu_\beta^2(K)}{(1\vee\beta)^2}
    +
    \frac{R(K)\sigma^2(x_0)} {p_X(x_0)\,nh^d}
    +\biggl(h^{2\beta}+\frac{1}{nh^d}\biggr).
\end{equation*}
Recalling~\eqref{eq:h-K} and noting $R(K)=\frac{2(\beta+d)\Gamma(1+d/2)}{\pi^{d/2}(2\beta+d)}$, and $\mu_\beta(K)=\frac{d}{2\beta+d}$,  
gives the required result;
\begin{align}
    \sup_{f\in\cH(\beta,L)}\E_{f}\bigl[(\hat{f}_{n,h}(x_0)&-f(x_0))^2\bigr]  
    \notag
    \\
    &=
    \biggl(\frac{d^d (\beta+d)^{2\beta}\Gamma^{2\beta}(1+d/2)}{\pi^{\beta d}\beta^{2\beta}(2\beta+d)^{d}(1\vee\beta)^{2d}}\biggr)^{1/(2\beta+d)}
    \biggl(\frac{L^{d/\beta}\sigma^2(x_0)}{p_X(x_0) n}\biggr)^{2\beta/(2\beta+d)}    
    (1+o(1)).
    \label{eq:HD-UB}
\end{align}
\end{proof}

\subsection{Proof of Theorem~\ref{thm:asymptotics}}

We first prove a finite sample lower bound on the pointwise mean squared error. 

\begin{lemma}[Finite sample lower bound for H\"older function estimation]\label{lem:finite}
    Consider the model~\eqref{eq:Y=f(X)+e} with $p_X(\cdot)/p_X(x_0)\in\cH(\alpha,L_X)$ for some $\alpha\in(0,1]$, $L_X>0$. Then for any $n\in\N, n\geq3$, $d\in\N$, $L>0$, $\sigma^2>0$, $x_0\in\R^d$,
    \begin{multline}
    \sup_{f\in\cH(\beta,L)}\E_{f}\big[(\hat{f}_n(x_0)-f(x_0))^2\big]
    \\\geq
    \underbar{A}\Bigl(\frac{2\beta}{2\beta+d}\Bigr)\biggl(\frac{1}{1+c_{n,d}}\biggr)\biggl(\frac{d^d (\beta+d)^{2\beta}\Gamma^{2\beta}(1+d/2)}{\pi^{\beta d}\beta^{2\beta}(2\beta+d)^{d}(1\vee\beta)^{2d}}\biggr)^{1/(2\beta+d)}
    \biggl(\frac{L^{d/\beta}\sigma^2}{p_X(x_0) n}\biggr)^{2\beta/(2\beta+1)},
\end{multline}
    where
    \begin{equation*}
        c_{n,d}:=
        L_X
        \biggl(\frac{(\beta+d)(2\beta+d)\Gamma(1+d/2)\sigma^2}{2\pi^{d/2}\beta^2p_X(x_0)L^2\hspace{0.02cm}n}\biggr)^{\alpha/(2\beta+d)},
    \end{equation*}
    and
    \begin{equation*}
        \underbar{A}(a) := \frac{1}{a^a(1-a)^{1-a}}\sup_{\substack{0<\lambda\leq 1\\m>0}}\frac{\lambda^a}{\bigl\{(m+1)_2F_1\bigl(-\frac{1}{2},\frac{m}{2},\frac{m}{2}+1;-\frac{\lambda}{m^2}\bigr)\bigr\}^2}.
    \end{equation*}
\end{lemma}

\begin{proof}[Proof of Lemma~\ref{lem:finite}]

The proof follows similar arguments to that of Lemma~\ref{lem:LB}. Consider the family of distributions $(P_t)_{t\in[-1,1]}$ given by
\begin{gather*}
    Y_i\given X_i\iid N\bigl(f_t(X_i),\,\sigma^2\bigr),
    \qquad
    f_t(x) := \frac{t L h_n^\beta }{1\vee\beta}K\Bigl(\frac{x-x_0}{h_n}\Bigr),
    \qquad
    K(u)=\bigl(1-\|u\|_2^\beta\bigr)_+\,
    ,
    \\
    R_{\beta,d} := \frac{2\pi^{d/2}\beta^2}{(\beta+d)(2\beta+d)\Gamma(1+d/2)},
    \qquad h_n := \biggl(\frac{(1\vee\beta)^2 \sigma^2 \lambda}{L^2p_X(x_0)R_{\beta,d}(1+\phi_n)n}\biggr)^{1/(2\beta+d)},
    \\
    \phi_n := 
    L_X
    \biggl(\frac{(1\vee\beta)^2\sigma^2\lambda}{L^2p_X(x_0)R_{\beta,d}n}\biggr)^{\alpha/(2\beta+d)},
\end{gather*}
for arbitrary $\lambda\in(0,1]$. 
As in the proof of Lemma~\ref{lem:LB}$, f_t\in\cH(\beta,L)$. 
The Fisher information is
\begin{equation*}
    \cI(t) 
    \leq
    \frac{L^2 n h_n^{2\beta+d}R_{\beta,d}\hspace{0.02cm}p_X(x_0)}{(1\vee\beta)^2\sigma^2}
    \bigl(1+h_n^{\alpha}L_X\bigr)
    =\frac{1+h_n^\alpha L_X}{1+\phi_n}\hspace{0.02cm}\lambda 
    \leq \lambda
    .
\end{equation*}
Again taking~$\hat{t}_n:=\frac{1\vee\beta}{Lh_n^{\beta}}\hat{f}_n(x_0)$,
\begin{align*}
    &\quad\sup_{f\in\cH(\beta,L)}\E_{f}\big[(\hat{f}_n(x_0)-f(x_0))^2\big]
    \\
    &\geq
    \frac{1}{(1\vee\beta)^2}\biggl(\frac{(1\vee\beta)^2 L^{d/\beta} \sigma^2}{p_X(x_0) R_{\beta,d} n }\cdot\frac{1}{1+c_{n,d}}\biggr)^{2\beta/(2\beta+d)}
    \biggl(\frac{(2\beta)^{2\beta}d^d}{(2\beta+d)^{2\beta+d}}\biggr)^{1/(2\beta+d)}\underbar{A}\biggl(\frac{2\beta}{2\beta+d}\biggr).
\end{align*}
The result follows.

\end{proof}

Theorem~\ref{thm:asymptotics} follows from Lemma~\ref{lem:finite}, as outlined below.

\begin{proof}[Proof of Theorem~\ref{thm:asymptotics}]
Under the assumed asymptotic regime,
\begin{align*}
    p_X(x_0)L^2R_{\beta,d}
    \geq
    \frac{L^2R_{\beta,d}}{C\,\mathrm{vol}\,\mathcal{B}_1(0)}
    =
    \frac{2\beta^2L^2}{C(\beta+d)(2\beta+d)}
    \sim 
    \frac{2\beta^2L^2}{Cd^2},
\end{align*}
and so
\begin{equation}
    c_{n,d}\sim L_X L^{-\frac{2\alpha}{d}} n^{-\frac{\alpha}{d}}
    =
    L_X e^{-2\alpha\frac{\log L}{d}} e^{-\alpha\frac{\log n}{d}} \to 0.
\end{equation}
Moreover, for $a\in(0,1/2)$,
\begin{align*}
        \underbar{A}(a)
        =
        \frac{1}{a^a(1-a)^{1-a}}\sup_{\substack{0<\lambda\leq 1\\m>0}}&\frac{4\lambda^a}{(m+1)^2 \bigl(\int_0^1 t^{\frac{m}{2}-1}\sqrt{m^2+\lambda t}\,dt\bigr)^2}
        \\
        &\geq
        \frac{1}{a^a(1-a)^{1-a}}\sup_{\substack{0<\lambda\leq 1\\m>0}}\frac{m^2\lambda^a}{(m+1)^2 (m^2+\lambda)} = a^{2a}(1-a)^{2(1-a)},
\end{align*}
with supremum attained at $(\lambda,m)=\bigl(\frac{a^3}{(1-a)^3}, \frac{a}{1-a}\bigr)$. Therefore, for $\beta\in(0,2]$, $d\geq 2$,
\begin{equation*}
    \underbar{A}\biggl(\frac{2\beta}{2\beta+d}\biggr) \geq
    \biggl(\frac{(2\beta)^{2\beta}d^d}{(2\beta+d)^{2\beta+d}}\biggr)^{2/(2\beta+d)} \sim 1,
\end{equation*}
as $d\to\infty$. 
As in the proof of Lemma~\ref{lem:UBB}, the asymptotic arguments also follow in the high-dimensional regime, and so~\eqref{eq:HD-UB} continues to hold. Therefore,
\begin{equation*}
        \lim_{n\to\infty}\frac{\inf_{\hat{f}}\sup_{f\in\cH(\beta,L)}\E_f\Bigl[\bigl(\hat{f}(x_0)-f(x_0)\bigr)^2\Bigr]}{\bigl(\frac{d^d (\beta+d)^{2\beta}\Gamma^{2\beta}(1+d/2)}{\pi^{\beta d}\beta^{2\beta}(2\beta+d)^{d}(1\vee\beta)^{2d}}\bigr)^{1/(2\beta+d)}\bigl(\frac{L^{d/\beta}\sigma^2}{p_X(x_0)n}\bigr)^{2\beta/(2\beta+d)}}
        =
        1,
    \end{equation*}
Applying Stirling's approximation,
\begin{equation*}
    \Gamma^{2\beta/(2\beta+d)}(1+d/2) \sim \biggl\{\sqrt{\pi d}\biggl(\frac{d}{2e}\biggr)^{d/2}\biggr\}^{2\beta/d}
    \sim \pi^{\beta/d}\bigl(d^{1/d}\bigr)^{\beta}\biggl(\frac{d}{2e}\biggr)^{\beta}
    \sim \biggl(\frac{d}{2e}\biggr)^{\beta}.
\end{equation*}
Therefore
\begin{equation*}
    \biggl(\frac{d^d (\beta+d)^{2\beta}\Gamma^{2\beta}(1+d/2)}{\pi^{\beta d}\beta^{2\beta}(2\beta+d)^{d}(1\vee\beta)^{2d}}\biggr)^{1/(2\beta+d)}
    \sim
    \frac{d^{2\beta/d}}{\pi^\beta (1\vee\beta)^2}\biggl(\frac{d}{2e}\biggr)^{\beta}
    \sim
    \frac{d^\beta}{(2\pi e)^\beta (1\vee\beta)^2},
\end{equation*}
from which the result follows.

\end{proof}

\subsection{Proof of Theorem~\ref{thm:AGVT}}

\begin{proof}[Proof of Theorem~\ref{thm:AGVT}]
    We generalize the proof of Theorem~\ref{thm:gvt}. Define
    \begin{equation*}
        \zeta(x,t) := \frac{\partial_t\bigl(p(x,t)\alpha(t)\bigr)}{p(x,t)\mu(t)}.
    \end{equation*}
    Consider the matrix
    \begin{align*}
        &\quad
        \begin{pmatrix}
            \int_T\int_{\cX}\bigl(\hat\psi(x)-\phi(t)\bigr)\bigl(\hat\psi(x)-\phi(t)\bigr)^\top p(x,t)\mu(t)\,d\nu(x)\,dt
            &
            \int_T\int_{\cX}\bigl(\hat\psi(x)-\phi(t)\bigr)\zeta(x,t)^\top p(x,t)\mu(t)\,d\nu(x)\,dt
            \\
            \int_T\int_{\cX}\zeta(x,t)\bigl(\hat\psi(x)-\phi(t)\bigr)^\top p(x,t)\mu(t)\,d\nu(x)\,dt
            &
            \int_T\int_{\cX}\zeta(x,t)\zeta(x,t)^\top p(x,t)\mu(t)\,d\nu(x)\,dt
        \end{pmatrix}
        \\
        &=
        \int_T\int_{\cX}
        \begin{pmatrix}
            \hat\psi(x)-\phi(t)
            \\
            \zeta(x,t)
        \end{pmatrix}
        \begin{pmatrix}
            \hat\psi(x)-\phi(t)
            \\
            \zeta(x,t)
        \end{pmatrix}^\top
        p(x,t)\mu(t)\,d\nu(x)\,dt
        \succcurlyeq 0,
    \end{align*}
    where $\succcurlyeq$ denotes the Loewner order. 
    Then, 
    \begin{align*}
        &\quad
        \int_T\int_{\cX}\bigl(\hat\psi(x)-\phi(t)\bigr)\bigl(\hat\psi(x)-\phi(t)\bigr)^\top p(x,t)\mu(t) \,d\nu(x)\,dt
        \\
        &\succcurlyeq
        \biggl(\int_T\int_{\cX}\zeta(x,t)\bigl(\hat\psi(x)-\phi(t)\bigr)^\top p(x,t)\mu(t)\,d\nu(x)\,dt\bigg)^\top
        \biggl(\int_T\int_{\cX}\zeta(x,t)\zeta(x,t)^\top p(x,t)\mu(t)\,d\nu(x)\,dt\biggr)^{-1}
        \\
        &\qquad
        \cdot\biggl(\int_T\int_{\cX}\zeta(x,t)\bigl(\hat\psi(x)-\phi(t)\bigr)^\top p(x,t)\mu(t)\,d\nu(x)\,dt\bigg).
    \end{align*}
    First, as $\alpha(t)=0$ on $\partial T$ the boundary of $T$,
    \begin{equation*}
        \int_T \zeta(x,t) p(x,t)\mu(t) \,dt = \int_T\partial_t\bigl(p(x,t)\alpha(t)\bigr)\,dt = 0
    \end{equation*}
    for $\nu$-almost all $x$. Thus
    \begin{equation}\label{eq:eq1}
        \int_T\int_{\cX}\hat\psi(x)p(x,t)\mu(t)\,d\nu(x)\,dt=0.
    \end{equation}
    By integration by parts
    \begin{equation*}
        \int_T \phi(t) \partial_t\bigl(p(x,t)\alpha(t)\bigr) \,dt
        =
        -\int_T p(x,t)\alpha(t)\nabla\phi(t)\,dt,
    \end{equation*}
    and so
    \begin{equation}\label{eq:eq2}
        \int_T\int_{\cX}\phi(t)\zeta(x,t)p(x,t)\mu(t)\,d\nu(x)\,dt
        =
        -\int_T  \alpha(t)\nabla\phi(t)\,dt.
    \end{equation}    
    Now,
    \begin{align}
        &\quad
        \int_T\int_{\cX}\zeta(x,t)\zeta(x,t)^\top p(x,t)\mu(t)\,d\nu(x)\,dt
        \notag
        \\
        &=
        \int_T\int_{\cX}\frac{\bigl(\partial_tp(x,t)\alpha(t)+p(x,t)\nabla \alpha(t)\bigr) \bigl(\partial_tp(x,t)\alpha(t)+p(x,t)\nabla \alpha(t)\bigr)^\top}{p(x,t)\mu(t)}\,d\nu(x)\,dt
        \notag
        \\
        &=
        \int_T\int_{\cX}\biggl(
        \frac{\partial_tp(x,t)\,\partial_tp(x,t)^\top}{p(x,t)}\frac{\alpha^2(t)}{\mu(t)}
        +
        \frac{\nabla \alpha(t)\nabla \alpha(t)^\top}{\mu(t)}p(x,t)
        \biggr)
        \,d\nu(x)\,dt
        \notag
        \\
        &=\int_T\int_{\cX}
        \biggl(\cI(t)\frac{\alpha^2(t)}{\mu(t)}
        +
        \frac{\nabla \alpha(t) \nabla \alpha(t)^\top}{\mu(t)}
        \biggr)\,dt
        ,
        \label{eq:eq3}
    \end{align}
    where in the second equality the cross terms are zero. By applying Jensen's inequality
    \begin{equation*}
        \int_{\cX}\|\partial_tp(x,t)\|_2\,d\nu(x) \leq \|\cI(t)\|_{\mathrm{op}}^{1/2} < \infty,
    \end{equation*}
    and so applying dominated convergence
    \begin{multline*}
        \int_{\cX} \frac{\bigl(\partial_tp(x,t)\,\alpha(t)\bigr)\bigl(p(x,t)\nabla \alpha(t)^\top\bigr)}{p(x,t)\mu(t)}\,d\nu(x)
        =
        \frac{\alpha(t)}{\mu(t)}\int_{\cX}\partial_tp(x,t)\,d\nu(x)\nabla \alpha(t)^\top
        \\
        =
        \frac{\alpha(t)}{\mu(t)}\biggl(\partial_t\int_{\cX}p(x,t)\,d\nu(x)\biggr)\nabla \alpha(t)^\top
        =0.
    \end{multline*}
    Combining~\eqref{eq:eq1},~\eqref{eq:eq2}~and~\eqref{eq:eq3} gives 
    \begin{multline*}
        \int_T\int_{\cX}\bigl(\hat\psi(x)-\phi(t)\bigr)\bigl(\hat\psi(x)-\phi(t)\bigr)^\top p(x,t)\mu(t) \,d\nu(x)\,dt
        \\
        \succcurlyeq
        \biggl(\int_T\nabla\phi(t)\alpha(t)\,dt\biggr)^\top
        \biggl(\int_T\frac{\cI(t)\alpha^2(t)+\nabla \alpha(t)\nabla \alpha(t)^\top}{\mu(t)}\,dt\biggr)^{-1}
        \biggl(\int_T\nabla\phi(t)\alpha(t)\,dt\biggr).
    \end{multline*}
    The result then follows by applying~\citet[Lemma 4]{gvt} (an application of the reverse triangle inequality).
    
\end{proof}

\subsection{Connections to~\citet{gvt}}\label{appsec:gvt}

In recent work~\citet{gvt} introduce the generalized van Trees inequality, which extends the van Trees inequality by the inclusion of an additional \emph{approximation function} $\phi:T\to\R$. This generalization allows one to obtain minimax lower bounds for non-differentiable
functionals and estimation problems in irregular statistical models. In the our univariate setting of Section~\ref{sec:gvt} their generalized van Trees inequality reduces to
\begin{equation*}
    \sup_{\phi}\left(\sqrt{\frac{\bigl(\int_T\phi'(t)\mu(t)\,dt\bigr)^2}{\int_T\cI(t)\mu(t)\,dt+\mathcal{J}(\mu)}}-\sqrt{\int_T\bigl(\phi(t)-t\bigr)^2\mu(t)\,dt}\right)_+^2.
\end{equation*}
As this isn't a setting to which one would expect the generalized van Trees approximation function, geared towards extending to allow for non-differentiable functionals, we expect in this specific setting the van Trees inequality to be recovered, i.e.
\begin{equation*}
    \sup_{\phi}\left(\sqrt{\frac{\bigl(\int_T\phi'(t)\mu(t)\,dt\bigr)^2}{\int_T\cI(t)\mu(t)\,dt+\mathcal{J}(\mu)}}-\sqrt{\int_T\bigl(\phi(t)-t\bigr)^2\mu(t)\,dt}\right)_+^2
    =
    \frac{1}{\int_T\cI(t)\mu(t)\,dt+\mathcal{J}(\mu)}
    ,
\end{equation*}
with supremum attained by taking $\phi=\mathrm{id}$. 
To see this, take $\phi(t)=t+\varphi(t)$ for some $\varphi:T\to\R$. Then
\begin{align*}
    &\quad\;
    \sqrt{\frac{\bigl(\int_T\phi'(t)\mu(t)\,dt\bigr)^2}{\int_T\cI(t)\mu(t)\,dt+\mathcal{J}(\mu)}}-\sqrt{\int_T\bigl(\phi(t)-t\bigr)^2\mu(t)\,dt}
    \\
    &=
    \frac{\bigl|\int_T\phi'(t)\mu(t)\,dt\bigr|}{\sqrt{\int_T\cI(t)\mu(t)\,dt+\mathcal{J}(\mu)}}-\sqrt{\int_T\bigl(\phi(t)-t\bigr)^2\mu(t)\,dt}
    \\
    &=
    \frac{\bigl|1-\int_T\varphi(t)\mu'(t)\,dt\bigr|}{\sqrt{\int_T\cI(t)\mu(t)\,dt+\mathcal{J}(\mu)}}-\sqrt{\int_T\bigl(\varphi(t)\bigr)^2\mu(t)\,dt}.
\end{align*}

Now, because $|1-x|\leq1+|x|$, and using the Cauchy--Schwarz inequality,
\begin{align*}
    \biggl|1-\int_T\varphi\mu'\biggr|
    \leq
    1+\biggl|\int_T\varphi\mu'\biggr|
    \leq
    1+\biggl(\int_T\varphi^2\mu\biggr)^{1/2}\biggl(\int_T\frac{(\mu')^2}{\mu}\biggr)^{1/2}
    =
    1+\sqrt{\int_T\varphi^2\mu\;}\sqrt{\mathcal{J}(\mu)}.
\end{align*}
Thus
\begin{align*}
    \sqrt{\frac{\bigl(\int_T\phi'(t)\mu(t)\,dt\bigr)^2}{\int_T\cI(t)\mu(t)\,dt+\mathcal{J}(\mu)}}-\sqrt{\int_T\bigl(\phi(t)-t\bigr)^2\mu(t)\,dt}
    &\leq
    \frac{1}{\sqrt{\int_T\cI\mu+\mathcal{J}(\mu)}}+
    \frac{\sqrt{\int_T\varphi^2\mu\;}\sqrt{\mathcal{J}(\mu)}}{\sqrt{\int_T\mathcal{I}\mu+\mathcal{J}(\mu)}}-\sqrt{\int_T\varphi^2\mu}
    \\
    &\leq 
    \frac{1}{\sqrt{\int_T\cI\mu+\mathcal{J}(\mu)}},
\end{align*}
and so
\begin{equation*}
    \sup_{\phi}\left(\sqrt{\frac{\bigl(\int_T\phi'(t)\mu(t)\,dt\bigr)^2}{\int_T\cI(t)\mu(t)\,dt+\mathcal{J}(\mu)}}-\sqrt{\int_T\bigl(\phi(t)-t\bigr)^2\mu(t)\,dt}\right)_+^2
    =
    \frac{1}{\int_T\cI(t)\mu(t)\,dt+\mathcal{J}(\mu)}.
\end{equation*}
This observation naturally separates the purposes of the \emph{approximation function} of~\citet{gvt} and our~\emph{augmentation function}. As we see in Section~\ref{sec:takatsu}, we can also incorporate \emph{both} the approximation and augmentation functions into an \emph{augmented generalized van Trees} inequality.

\end{document}